\documentclass[11pt]{article}
\usepackage[english]{babel}
\usepackage{amssymb,amsmath,amsthm}
\newtheorem{theorem}{Theorem}[section]
\newtheorem{lemma}[theorem]{Lemma}

\newtheorem{corollary}[theorem]{Corollary}

{\theoremstyle{definition}}
{\theoremstyle{definition}}
{\theoremstyle{definition}\newtheorem{definition}[theorem]{Definition}}
{\theoremstyle{definition}\newtheorem{remark}[theorem]{Remark}}

\newtheorem*{thmhl}{Theorem HL}
\newtheorem*{thmbm}{Theorem BM}

\numberwithin{equation}{section}

\def\C{{\mathbb C}}

\def\N{{\mathbb N}}
\def\Z{{\mathbb Z}}
\def\R{{\mathbb R}}

\def\K{{\mathbb K}}

\def\LL{{\mathcal L}}

\def\rr{{\mathcal R}}
\def\pp{{\mathcal P}}
\def\kk{{\bf k}}
\def\epsilon{\varepsilon}
\def\kappa{\varkappa}
\def\phi{\varphi}
\def\leq{\leqslant}
\def\geq{\geqslant}

\def\ker{\hbox{\tt ker}\,}

\def\spann{\hbox{\tt span}\,}

\def\deg{\hbox{\tt deg}\,}

\title{Mixing operators on spaces with weak topology}

\author{Stanislav Shkarin}

\date{}

\begin{document}

\maketitle

\begin{abstract} We prove that a continuous linear operator $T$ on a
topological vector space $X$ with weak topology is mixing if and
only if the dual operator $T'$ has no finite dimensional invariant
subspaces. This result implies the characterization of hypercyclic
operators on the space $\omega$ due to Herzog and Lemmert and
implies the result of Bayart and Matheron, who proved that for any
hypercyclic operator $T$ on $\omega$, $T\oplus T$ is also
hypercyclic.
\end{abstract}



\section{Introduction \label{s1}}\rm

All topological vector spaces in this article {\it are assumed to be
Hausdorff} and are over the field $\K$, being either the field $\C$
of complex numbers or the field $\R$ of real numbers. As usual, $\Z$
is the set of integers and $\N$ is the set of positive integers. If
$X$ and $Y$ are vector spaces over the same field $\kk$, symbol
$L(X,Y)$ stands for the space of $\kk$-linear maps from $X$ to $Y$.
If $X$ and $Y$ are topological vector spaces, then $\LL(X,Y)$ is the
space of continuous linear operators from $X$ to $Y$. We write
$L(X)$ instead of $L(X,X)$, $\LL(X)$ instead of $\LL(X,X)$ and $X'$
instead of $\LL(X,\K)$. For each $T\in \LL(X)$, the dual operator
$T':X'\to X'$ is defined as usual: $(T'f)(x)=f(Tx)$ for $f\in X'$
and $x\in X$. We say that the topology $\tau$ of a topological
vector space $X$ {\it is weak} if $\tau$ is exactly the weakest
topology making each $f\in Y$ continuous for some linear space $Y$
of linear functionals on $X$ separating points of $X$. We use symbol
$\omega$ to denote the product of countably many copies of $\K$. It
is easy to see that $\omega$ is a separable complete metrizable
topological vector space, whose topology is weak.

Let $X$ be a topological vector space and $T\in\LL(X)$. A vector
$x\in X$ is called a {\it hypercyclic vector} for $T$ if
$\{T^nx:n\in\N\}$ is dense in $X$ and  $T$ is called {\it
hypercyclic} if it has a hypercyclic vector. Recall also that $T$ is
called {\it hereditarily hypercyclic} if for each infinite subset
$A$ of $\N$, there is $x\in X$ such that $\{T^nx:n\in A\}$ is dense
in $X$. Next, $T$ is called {\it transitive} if for any non-empty
open subsets $U$ and $V$ of $X$, there is $n\in\N$ for which
$T^n(U)\cap V\neq\varnothing$ and $T$ is called {\it mixing} if for
any non-empty open subsets $U$ and $V$ of $X$, there is $n\in\N$
such that $T^m(U)\cap V\neq\varnothing$ for each $n\geq m$. It is
well-known and easy to see that any hypercyclic operator (on any
topological vector space) is transitive and any hereditarily
hypercyclic operator is mixing. If $X$ is complete separable and
metrizable, then the converse implications hold: any transitive
operator is hypercyclic and any mixing operator is hereditarily
hypercyclic. For the proof of these facts as well as for any
additional information on the above classes of operators we refer to
the book \cite{bama-book} and references therein. Herzog and Lemmert
\cite{gele} characterized hypercyclic operators on $\omega$.

\begin{thmhl}Let $\K=\C$ and $T\in \LL(\omega)$. Then $T$ is
hypercyclic if and only if the point spectrum $\sigma_p(T')$ of $T'$
is empty.
\end{thmhl}

Another result concerning hypercyclic operators on $\omega$ is due
to Bayart and Matheron \cite{bama}.

\begin{thmbm}For any hypercyclic operator $T\in \LL(\omega)$, $T\oplus
T$ is also hypercyclic.
\end{thmbm}

We refer to \cite{beco,peter1} for results on the structure of the
set of hypercyclic vectors of operators on $\omega$ and to
\cite{chan,shk1} for results on hypercyclicity of operators on
Banach spaces endowed with its weak topology. We characterize
transitive and mixing operators on spaces with weak topology.

\begin{theorem}\label{omegA} Let $X$ be a topological vector space,
whose topology is weak and $T\in \LL(X)$. Then the following
conditions are equivalent$:$
\begin{itemize}\itemsep=-2pt
\item[\rm(\ref{omegA}.1)]$T'$ has no non-trivial finite dimensional invariant subspaces$;$
\item[\rm(\ref{omegA}.2)]$T$ is transitive$;$
\item[\rm(\ref{omegA}.3)]$T$ is mixing$;$
\item[\rm(\ref{omegA}.4)]for any non-empty open subsets $U$ and $V$
of $X$, there is $k\in\N$ such that $p(T)(U)\cap V\neq\varnothing$
for any polynomial $p$ of degree $\geq k$.
\end{itemize}
\end{theorem}

Since $\omega$ is complete, separable, metrizable and carries weak
topology, we obtain the following corollary.

\begin{corollary}\label{omeg} Let  $T\in \LL(\omega)$. Then the following conditions
are equivalent
\begin{itemize}\itemsep=-2pt
\item[\rm(\ref{omeg}.1)]$T'$ has no non-trivial finite dimensional invariant subspaces$;$
\item[\rm(\ref{omeg}.2)]$T$ is hypercyclic$;$
\item[\rm(\ref{omeg}.3)]$T$ is hereditarily hypercyclic$;$
\item[\rm(\ref{omeg}.4)]for any sequence $\{p_k\}_{k\in\N}$ of
polynomials with $\deg p_k\to\infty$, there is $x\in \omega$ such
that $\{p_k(T)x:k\in\N\}$ is dense in $\omega$.
\end{itemize}
\end{corollary}

Note that in the case $\K=\C$, $T'$ has no non-trivial finite
dimensional invariant subspaces if and only if
$\sigma_p(T')=\varnothing$. Moreover, the direct sum of two mixing
operators is always mixing. Thus Theorem~HL and Theorem~BM follow
from Theorem~\ref{omegA}.

\begin{remark}\label{Rem2} Chan and Sanders \cite{chan} observed
that on $(\ell_2)_\sigma$, being $\ell_2$ with the weak topology,
there is a transitive non-hypercyclic operator. Theorem~\ref{omegA}
provides a huge supply of such operators. For instance, the backward
shift $T$ on $\ell_2$ is mixing on $(\ell_2)_\sigma$ since $T'$ has
no non-trivial finite dimensional invariant subspaces and $T$ is
non-hypercyclic since each its orbit is bounded.
\end{remark}

Since each weak topology is determined by the corresponding space of
linear functionals, it comes as no surprise that Theorem~\ref{omegA}
is algebraic in nature. Indeed, we derive it from the following
characterization of linear maps without finite dimensional invariant
subspaces. The idea of the proof is close to that Herzog and Lemmert
\cite{gele}, although by reasoning on a more abstract level, we were
able to get a result, which is simultaneously stronger and more
general.

We start by introducing some notation. Let $\kk$ be a field. Symbol
$\pp$ stands for the algebra $\kk[t]$ of polynomials in one variable
over $\kk$, while $\rr$ is the field $\kk(t)$ of rational functions
in one variable over $\kk$. Consider the $\kk$-linear map
\begin{equation*}
M:\rr\to\rr,\qquad Mf(z)=zf(z).
\end{equation*}
If $A$ is a set and $X$ is a vector space, then $X^{(A)}$ stands for
the algebraic direct sum of copies of $X$ labeled by $A$:
$$
X^{(A)}=\bigoplus_{\alpha\in A}\rr=\bigl\{x\in X^A:\{\alpha\in
A:x_\alpha\neq 0\}\ \ \text{is finite}\bigr\}.
$$
Symbol $M^{(A)}$ stands for the linear operator on $\rr^{(A)}$,
being the direct sum of copies of $M$ labeled by $A$. That is,
$$
M^{(A)}\in L(\rr^{(A)}),\quad (M^{(A)}f)_\alpha=Mf_\alpha\ \ \
\text{for each}\ \ \alpha\in A.
$$
It is easy to see that each $M^{(A)}$ has no non-trivial finite
dimensional invariant subspaces. Obviously, the same holds true for
each restriction of $M^{(A)}$ to an invariant subspace.

\begin{theorem}\label{inva} Let $X$ be a vector space
over a field $\kk$ and $T\in L(X)$. Then $T$ has no non-trivial
finite dimensional invariant subspaces if and only if $T$ is similar
to a restriction of some $M^{(A)}$ to an invariant subspace.
\end{theorem}

The above theorem is interesting on its own right. It also allows us
to prove the following lemma, which is the key ingredient in the
proof of Theorem~\ref{omegA}.

\begin{lemma} \label{empty1}
Let $X$ be a non-trivial vector space over a field $\kk$ and $T:X\to
X$ be a linear map with no non-trivial finite dimensional invariant
subspaces. Then for any finite dimensional subspace $L$ of $X$,
there is $m=m(L)\in \N$ such that $p(T)(L)\cap L=\{0\}$ for each
$p\in\pp$ with $\deg p\geq m$.
\end{lemma}

\section{Linear maps without finite dimensional invariant subspaces}

Throughout this section $\kk$ is a field, $X$ is a non-trivial
linear space over $\kk$ and $T:X\to X$ is a $\kk$-linear map. We
also denote $\pp^*=\pp\setminus\{0\}$.

\begin{lemma} \label{inj} Let $T$ be a linear operator on a linear space
$X$. Then $T$ has no non-trivial finite dimensional invariant
subspaces if and only if $p(T)$ is injective for any non-zero
polynomial $p$.
\end{lemma}

\begin{proof} If $p$ is a non-zero polynomial and $p(T)$ is
non-injective, then there is non-zero $x\in X$ such that $p(T)x=0$.
Let $k=\deg p$. It is straightforward to verify that
$E=\spann\{x,Tx,\dots,T^{k-1}x\}$ is a non-trivial finite
dimensional invariant subspace for $T$. Assume now that $T$ has a
non-trivial finite dimensional invariant subspace $L$ and $p$ is the
characteristic polynomial of the restriction of $T$ to $L$. By the
Hamilton--Cayley theorem, $p(T)$ vanishes on $L$. Hence $p(T)$ is
non-injective.
\end{proof}

\begin{definition}\label{Tind}
For a linear operator $T$ on a vector space $X$ we say that vectors
$x_1,\dots,x_n$ in $X$ are $T$-{\it independent} if for any
polynomials $p_1,\dots,p_n$, the equality
$p_1(T)x_1+{\dots}+p_n(T)x_n=0$ implies $p_j=0$ for $1\leq j\leq n$.
Otherwise, we say that $x_1,\dots,x_n$ are $T$-{\it dependent}. A
set $A\subset X$ is called $T$-{\it independent} if any pairwise
different vectors $x_1,\dots,x_n\in A$ are $T$-{\it independent}.
\end{definition}

For a subset $A$ of a vector space $X$ and $T\in L(X)$, we denote
\begin{equation}\label{efat}
E(A,T)=\spann\biggl(\bigcup_{n=0}^\infty T^n(A)\biggr)\ \
\text{and}\ \ F(A,T)=\bigcup_{p\in\pp^*} p(T)^{-1}(E(A,T)).
\end{equation}
Clearly, $E(A,T)$ is the smallest subspace of $X$, containing $A$
and invariant with respect to $T$ and $F(A,T)$ consists of all $x\in
X$ for which
\begin{equation}\label{xat}
q(T)x=\sum_{a\in A} p_a(T)a
\end{equation}
for some $q\in\pp^*$ and $p=\{p_a\}_{a\in A}\in \pp^{(A)}$. Since
$p_a\neq 0$ for finitely many $a\in A$ only, the sum in the above
display is finite.

\begin{lemma}\label{tin} Let $T\in L(X)$ be a linear operator with
no non-trivial finite dimensional invariant subspaces and $A\subset
X$ be a $T$-independent set. Then $F(A,T)$ is a linear subspace of
$X$ invariant for $T$ and for every $x\in F(A,T)$, the rational
functions $f_{x,a}=\frac{p_a}{q}$ with $p\in \pp^{(A)}$ and
$q\in\pp^*$ satisfying $(\ref{xat})$ are uniquely determined by $x$
and $a\in A$. Moreover, the map
\begin{equation}\label{J}
J:F(A,T)\to \rr^{(A)},\quad J_x=\{f_{x,a}\}_{a\in A}
\end{equation}
is linear, injective and satisfies $JTx=M^{(A)}Jx$ for any $x\in
F(A,T)$. In particular, the restriction $T_A=T\bigr|_{F(A,T)}\in
L(F(A,T))$ is similar to the restriction of $M^{(A)}$ to the
invariant subspace $J(F(A,T))$.
\end{lemma}

\begin{proof} First, we show that the rational functions
$f_{x,a}=\frac{p_a}q$ for $a\in A$ are uniquely determined by $x\in
F(A,T)$. Assume that $q_1,q_2\in \pp^*$ and $\{p_{1,a}\}_{a\in
A},\{p_{2,a}\}_{a\in A}\in \pp^{(A)}$ are such that
$$
q_1(T)x=\sum_{a\in A} p_{1,a}(T)a\quad\text{and}\quad
q_2(T)x=\sum_{a\in A} p_{2,a}(T)a.
$$
Applying $q_2(T)$ to the first equality and $q_1(T)$ to the second,
we get
$$
(q_1q_2)(T)x=\sum_{a\in A} (q_2p_{1,a})(T)a=\sum_{a\in A}
(q_1p_{2,a})(T)a.
$$
Since $A$ is $T$-independent, $q_2p_{1,a}=q_1p_{2,a}$ for each $a\in
A$. That is, $\frac{p_{1,a}}{q_1}=\frac{p_{2,a}}{q_2}$. Thus the
rational functions $f_{x,a}=\frac{p_a}q$ for $a\in A$ are uniquely
determined by $x$. It is also clear that the set $\{a\in
A:f_{x,a}\neq 0\}$  is finite for each $x\in X$. Thus the formula
(\ref{J}) defines a map $J:F(A,T)\to \rr^{(A)}$.

Our next step is to show that $F(A,T)$ is a linear subspace of $X$
and that the map $J$ is linear. Let $x,y\in F(A,T)$ and $t,s\in\kk$.
Pick $q_1,q_2\in \pp^*$ and $\{p_{1,a}\}_{a\in A},\{p_{2,a}\}_{a\in
A}\in \pp^{(A)}$ such that
$$
q_1(T)x=\sum_{a\in A} p_{1,a}(T)a\quad\text{and}\quad
q_2(T)y=\sum_{a\in A} p_{2,a}(T)a.
$$
Hence
$$
(q_1q_2)(T)(tx+sy)=\sum\limits_{a\in B}
(tp_{1,a}q_2+sp_{2,a}q_1)(T)a.
$$
It follows that $tx+sy\in F(A,T)$ and therefore $F(A,T)$ is a linear
subspace of $X$. Moreover, by definition of the rational functions
$f_{x,a}$, we have $f_{x,a}=\frac{p_{1,a}}{q_1}$,
$f_{y,a}=\frac{p_{2,a}}{q_2}$ and
$$
f_{tx+sy,a}=\frac{tp_{1,a}q_2+sp_{2,a}q_1}{q_1q_2}=tf_{x,a}+sf_{y,a}\quad
\text{for any $a\in A$,}
$$
which proves linearity of $J$. Since $Jx=0$ if and only if $q(T)x=0$
for some $q\in\pp^*$, Lemma~\ref{inj} implies that $\ker J=\{0\}$.
That is, $J$ is injective.

Now let us show that $F(A,T)$ is invariant for $T$ and that
$JTx=M^{(A)}Jx$ for any $x\in F(A,T)$. Let $x\in F(A,T)$ and $q\in
\pp^*$ and $\{p_{a}\}_{a\in A}\in \pp^{(A)}$ be such that
$$
q(T)x=\sum_{a\in A} p_a(T)a.\ \ \ \text{Then}\ \ \
q(T)(Tx)=\sum\limits_{a\in B}p_{1,a}(T)a,\ \ \text{where
$p_{1,a}(z)=zp_a(z)$.}
$$
Hence $Tx\in F(A,T)$ and therefore $F(A,T)$ is invariant for $T$.
Moreover, $f_{Tx,a}=\frac{p_{1,a}}q=Mf_{x,a}$ for any $x\in F(A,T)$
and $a\in A$. That is, $JTx=M^{(A)}Jx$ for any $x\in F(A,T)$. Since
$J$ is injective it is a linear isomorphism of $F(A,T)$ and
$Y=J(F(A,T))$. Then the equality $JTx=M^{(A)}Jx$ for $x\in F(A,T)$
implies that $Y$ is invariant for $M^{(A)}$ and that $T_A$ is
similar to $M^{(A)}\bigr|_{Y}$ with the linear map $J$ providing the
similarity.
\end{proof}

\subsection{Proof of Theorem~\ref{inva}}

Let $T\in L(X)$ be without non-trivial finite dimensional invariant
subspaces. A standard application of the Zorn lemma allows us to
take a maximal by inclusion $T$-independent subset $A$ of $X$. From
the definition of the spaces $F(B,T)$ it follows that if $B\subset
X$ is $T$-independent and $x\in X\setminus F(B,T)$, then
$B\cup\{x\}$ is also $T$ independent. Thus maximality of $A$ implies
that $X=F(A,T)$. By Lemma~\ref{tin}, $T=T_A$ is similar to a
restriction of $M^{(A)}$ to an invariant subspace.

\subsection{Proof of Lemma~\ref{empty1}}

Let $A\subset L$ be a linear basis of $L$. Since $L$ is finite
dimensional, $A$ is finite. Pick a maximal by inclusion
$T$-independent subset $B$ of $A$ (since $A$ is finite, we do not
need the Zorn lemma to do that). Now let $F=F(B,T)$ be the subspace
of $X$ defined in (\ref{efat}). Since $B$ is a maximal
$T$-independent subset of a basis of $L$, $L\subseteq F$. By
Lemma~\ref{tin}, $F$ is invariant for $T$. Thus we can without loss
of generality assume that $X=F$. Then by Lemma~\ref{tin}, we can
assume that $T$ is a restriction of $M^{(B)}$ to an invariant
subspace. Since extending $T$ beyond $X$ is not going to change the
spaces $L\cap p(T)(L)$, we can assume that $T=M^{(B)}$. Since $B$ is
finite, without loss of generality, $X=\rr^n$ and
$T=M\oplus{\dots}\oplus M$, where $n\in\N$.

Consider the degree function $\deg:\rr\to\Z\cup\{-\infty\}$. We set
$\deg(0)=-\infty$ and let $\deg(p/q)=\deg p-\deg q$, where $p$ and
$q$ are non-zero polynomials and the degrees in the right hand side
are the conventional degrees of polynomials. Clearly this function
is well-defined and is a grading on $\rr$. That is,
\begin{itemize}\itemsep=-2pt
\item[(g1)]$\deg(f_1f_2)=\deg(f_1)\!+\!\deg(f_2)$ and $\deg(f_1\!+\!f_2)\leq \max\{\deg f_1,\deg
f_2\}$ for any $f_1,f_2\in\rr$;
\item[(g2)]if $f_1,f_2\in\rr$ and $\deg f_1\neq \deg f_2$, then
$\deg(f_1+f_2)=\max\{\deg f_1,\deg f_2\}$.
\end{itemize}
By (g1), $\deg(Mf)=1+\deg f$ for each $f\in\rr$. For $f\in X=\rr^n$,
we write
$$
\delta(f)=\max_{1\leq j\leq n}\deg f_j.
$$
Clearly $\delta(0)=-\infty$ and $\delta(f)\in\Z$ for each $f\in
X\setminus \{0\}$. Let also
$$
\Delta^+=\sup_{f\in L} \delta(f)\ \ \text{and}\ \
\Delta^-=\inf_{f\in L\setminus\{0\}} \delta(f).
$$
Then $\Delta^+$ and $\Delta^-$ are finite. Indeed, assume that
either $\Delta^+=+\infty$ or $\Delta^-=-\infty$. Then there exists a
sequence $\{u_l\}_{l\in\N}$ in $L\setminus\{0\}$ such that
$\{\delta(u_l)\}_{l\in\N}$ is strictly monotonic. For each $l$ we
can pick $j(l)\in\{1,\dots,n\}$ such that $\delta(u_l)=\deg
(u_l)_{j(l)}$. Then there is $\nu\in \{1,\dots,n\}$ for which the
set $B_\nu=\{l\in\N:j(l)=\nu\}$ is infinite. It follows that the
degrees of $(u_l)_\nu$ for $l\in B_\nu$ are pairwise different.
Property (g2) of the degree function implies that the rational
functions $(u_l)_\nu$ for $l\in B_\nu$ are linearly independent.
Hence the infinite set $\{u_l:l\in B_\nu\}$ is linearly independent
in $X$, which is impossible since all $u_l$ belong to the finite
dimensional space $L$. Thus $\Delta^+$ and $\Delta^-$ are finite.

Now let $p\in\pp^*$ and $d=\deg p$. By (g1) and the equality
$(Tf)_j=Mf_j$, $\delta(p(T)f)=\delta(f)+d$ for each $f\in X$.
Therefore, $\inf\bigl\{\delta(f):f\in
p(T)(L)\setminus\{0\}\bigr\}=\Delta^-+d$. In particular, if
$d>\Delta^+-\Delta^-$, then
$$
\inf_{f\in
p(T)(L)\setminus\{0\}}\delta(f)=\Delta^-+d>\Delta^+=\sup_{f\in L}
\delta(f).
$$
Thus $\delta(u)>\delta(v)$ for any non-zero $u\in P(T)(L)$ and $v\in
L$, which implies that $p(T)(L)\cap L=\{0\}$ whenever $\deg
p>\Delta^+-\Delta^-$. Thus the number $m=\Delta^+-\Delta^-+1$
satisfies the desired condition. The proof of Lemma~\ref{empty1} is
complete.

\section{Proof of Theorem~\ref{omegA}}

The implications
$(\ref{omegA}.4)\Longrightarrow(\ref{omegA}.3)\Longrightarrow(\ref{omegA}.2)$
are trivial. Assume that $T$ is transitive and $T'$ has a
non-trivial finite dimensional invariant subspace. Then $T$ has a
non-trivial closed invariant subspace of finite codimension. Passing
to the quotient by this subspace, we obtain a transitive operator on
a finite dimensional topological vector space. Since there is only
one Hausdorff vector space topology on a finite dimensional space,
we arrive to a transitive operator on a finite dimensional Banach
space. Since transitivity and hypercyclicity for operators on
separable Banach spaces are equivalent, we obtain a hypercyclic
operator on a finite dimensional Banach space. On the other hand, it
is well known that such operators do not exist, see, for instance,
\cite{ww}. Thus (\ref{omegA}.2) implies (\ref{omegA}.1). It remains
to show that (\ref{omegA}.1) implies (\ref{omegA}.4).

Assume that (\ref{omegA}.1) is satisfied and (\ref{omegA}.4) fails.
Then there exist non-empty open subsets $U$ and $V$ of $X$ and a
sequence $\{p_l\}_{l\in\N}$ of polynomials such that $\deg
p_l\to\infty$ and $p_l(T)(U)\cap V=\varnothing$ for each $l\in\N$.
Since $X$ carries weak topology, there exist two finite linearly
independent sets $\{f_1,\dots,f_n\}$ and $\{g_1,\dots,g_m\}$ in $X'$
and two vectors $(a_1,\dots,a_n)\in\K^n$ and
$(b_1,\dots,b_m)\in\K^m$ such that $U_0\subseteq U$ and
$V_0\subseteq V$, where
\begin{equation*}
U_0=\{u\in X:f_k(u)=a_k\ \ \text{for}\ \ 1\leq k\leq n\}\ \
\text{and}\ \ V_0=\{u\in X:g_j(u)=b_j\ \ \text{for}\ \ 1\leq j\leq
m\}.
\end{equation*}
Let $L=\spann\{f_1,\dots,f_n,g_1,\dots,g_m\}$. Since $T'$ has no
non-trivial finite dimensional invariant subspaces, by
Lemma~\ref{empty1}, $p_l(T')(L)\cap L=\{0\}$ for any sufficiently
large $l$. For such an $l$, the equality $p_l(T')(L)\cap L=\{0\}$
together with the injectivity of $p_l(T')$, provided by
Lemma~\ref{inj}, and the definition of $L$ imply that the vectors
$p_l(T')g_1,\dots,p_l(T')g_m,f_1,\dots,f_n$ are linearly
independent. Hence there exists $u\in X$ such that
$$
\text{$p_l(T')g_j(u)=b_j$ \ for \ $1\leq j\leq m$\ \ and\ \
$f_k(u)=a_k$ \ for $1\leq k\leq n$.}
$$
Since $p_l(T')g_j(u)=g_j(p_l(T)u)$, the last display implies that
$u\in U_0\subseteq U$ and $p_l(T)u\in V_0\subseteq V$. Hence
$p_l(T)(U)\cap V$ contains $p_l(T)u$ and therefore is non-empty.
This contradiction completes the proof of Theorem~\ref{omegA}.

\begin{remark}\label{last} The only place in the proof of
Theorem~\ref{omegA}, where we used the nature of the underlying
field, is the reference to the absence of transitive operators on
non-trivial finite dimensional spaces. Thus Theorem~\ref{omegA}
extends to topological vector spaces with weak topology over any
topological field $\kk$ provided there are no transitive operators
on non-trivial finite dimensional topological vector spaces over
$\kk$.
\end{remark}

The author would like to thank the referee for helpful comments. 

\small\rm

\vskip1truecm

\scshape

\noindent Stanislav Shkarin

\noindent Queens's University Belfast

\noindent Department of Pure Mathematics

\noindent University road, Belfast, BT7 1NN, UK

\noindent E-mail address: \qquad {\tt s.shkarin@qub.ac.uk}

\end{document}